\documentclass[12pt]{amsart}
\usepackage{amsmath,amsthm,amssymb,mathtools,enumitem}
\numberwithin{equation}{section}
\newtheorem{thm}{Theorem}[section]
\newtheorem{prop}[thm]{Proposition}
\newtheorem{lem}[thm]{Lemma}
\newtheorem{cor}[thm]{Corollary}

\theoremstyle{definition}

\theoremstyle{remark}
\newtheorem{remark}[thm]{Remark}

\begin{document}
	
	\title[Target Avoidance for Manneville--Pomeau Maps]
	{Strong-Winning Target Avoidance for Manneville--Pomeau Maps}
	\author{Jason Duvall}
	\email{jasonrduvall@gmail.com}
	\thanks{Keywords: Schmidt's game, Manneville--Pomeau maps, nondense orbit, nonuniform hyperbolicity, Hausdorff dimension, exceptional sets}
	\thanks{Mathematics Subject Classification numbers: 37D25, 11K55}
	\begin{abstract}
		We prove that target-avoidance sets for Manneville--Pomeau maps are strong winning for Schmidt's game. More precisely, for the class of nonuniformly expanding interval maps considered here, there exists a single parameter $\alpha>0$ such that for every target $p\in[0,1]$, the set of points whose forward orbit does not accumulate on $p$ is $\alpha$-strong winning. The proof induces on the uniformly expanding region $[r_1,1]$. The resulting first-return map has infinitely many branches, so we approximate it by finite-branch expanding maps, apply a theorem of Hu--Li--Yu to those finite approximants, and then transfer the resulting strategies first to the induced map and then to the original Manneville--Pomeau map.
	\end{abstract}
	\maketitle
	
	\section{Introduction}
	
	\subsection{Background}
	
	Manneville--Pomeau maps are standard examples of nonuniformly expanding interval maps. They are expanding away from a neutral fixed point, while orbits passing near the neutral fixed point can spend long periods of time in a region of weak expansion. For this reason they provide a useful test case for determining which conclusions from uniformly expanding dynamics survive in the presence of intermittent behavior.
	
	In this paper we study target-avoidance sets for such maps. Given a point $p\in[0,1]$, define
	\begin{equation*}
		\mathcal E_f(p)
		\coloneqq
		\big\{x\in[0,1]:p\notin\overline{\{f^n x:n\geq0\}}\big\}.
	\end{equation*}
	Equivalently, $\mathcal E_f(p)$ is the set of points whose forward orbit misses some relative neighborhood of $p$. In many ergodic systems these sets are small from the viewpoint of invariant measure, but they are often large from the viewpoint of Hausdorff dimension and Schmidt games.
	
	Schmidt's game, introduced in \cite{Schmidt1966}, and later variants such as McMullen's strong and absolute games \cite{McMullen2010}, give robust ways to measure this kind of largeness. Strong-winning sets have full Hausdorff dimension and satisfy a countable-intersection property with a fixed winning parameter, and the strong winning (but not the winning property) is stable under quasisymmetric homeomorphisms. Related work on nondense orbits and Schmidt games includes \cite{Dani1988,BroderickFishmanKleinbock2011,Tseng2009,Farm2010,FarmPerssonSchmeling2010,ManceTseng2013,HuYu,HuLiYu2025}.
	
	For uniformly expanding interval maps, target-avoidance results are now well developed. Hu and Yu proved $1/2$-winning results for a class of expanding interval maps including the Gauss transformation and beta transformations \cite{HuYu}. More recently, Hu, Li, and Yu proved that if $T$ is a piecewise $C^{1+\delta}$ expanding interval map with finitely many branches, then for every target $\xi$ the exceptional set
	\begin{equation*}
		BAD_T(\xi)=\{x:\inf_{n\ge0}|T^n x-\xi|>0\}
	\end{equation*}
	is $1/2$-strong winning \cite{HuLiYu2025}. This finite-branch theorem is a key ingredient below.
	
	The nonuniformly expanding Manneville--Pomeau case is not a direct consequence of the finite-branch theory. The natural first-return map to the uniformly expanding region $[r_1,1]$ has infinite branches and unbounded geometry near the neutral fixed point, complicating the analysis and preventing direct application of \cite{HuYu} or \cite{HuLiYu2025}.
	
	\subsection{Our results}
	
	In \cite{DuvallMP}, the author proved that the exceptional set based at the neutral fixed point is strong winning. The purpose of the present paper is to extend this from the neutral fixed point to arbitrary targets, with one uniform winning parameter for all points in the unit interval.
	
	The proof proceeds as follows. We induce on the uniformly expanding region $[r_1,1]$, obtaining a countably branched full-return map $F$. For each $K$ we form a finite-branch expanding approximation $F_K$ by keeping the first $K$ return branches and replacing the remaining tail by a single affine branch. The theorem of Hu--Li--Yu applies to each $F_K$. We then show that the intersection of the corresponding finite-branch exceptional sets, together with the known exceptional set for the neutral endpoint of the induced system, is contained in the desired exceptional set for $F$. Finally, the inducing transfer argument from \cite{DuvallMP} carries strong winning from the induced system back to the original map.
	
	The main result is the following.
	
	\begin{thm}\label{thm:maintheorem}
		Let $f$ satisfy Conditions \ref{condfirst}--\ref{condlast}. There exists $\alpha\in\bigl(0,\frac12\bigr)$ such that for every $p\in[0,1]$, the exceptional set $\mathcal E_f(p)$ is $\alpha$-strong winning for Schmidt's game on $[0,1]$.
	\end{thm}
	
	Because countable intersections of $\alpha$-strong winning sets are again $\alpha$-strong winning, we immediately obtain the following consequence.
	
	\begin{cor}
		Let $P\subset[0,1]$ be countable. Then the set of points $x\in[0,1]$ such that
		\begin{equation*}
			P\cap\overline{\{f^n x:n\geq0\}}=\emptyset
		\end{equation*}
		is $\alpha$-strong winning.
	\end{cor}
	
	In particular, the set of points whose orbit closures avoid a prescribed countable set, such as a countable collection of periodic orbits or $\mathbb Q\cap[0,1]$, is strong winning.
	
	\section{Preliminaries}\label{sec:preliminaries}
	
	For $I \subset \mathbb{R}$ we denote the left and right endpoints $I$ by $\partial^\ell I, \partial^r I$, respectively.
	
	\subsection{The class of maps under consideration}
	
	Throughout this paper we assume that $f \colon [0,1] \to [0,1]$ is a two-branched interval map and that there exists $r_1 \in (0,1)$ such that the following conditions hold (using one-sided derivatives as needed):
	\begin{enumerate}
		\item \label{condfirst} $f((0,r_1)) = f((r_1,1)) = (0,1)$;
		\item \label{cond2} $f$ is $C^1$ on $[0,r_1)$ and $C^2$ on $(0,r_1) \cup (r_1,1]$;
		\item \label{cond3} $f(0) = f(r_1) = 0$, $f(1)=1$, $f'(0) = 1$, and $\lvert f' \rvert > 1$ on $(0,1]$;
		\item \label{cond4} $\sup_{x \in [0,1]} \, \lvert f'(x) \rvert < \infty$ and 
		$\sup_{x \in (r_1,1]} \, \lvert f''(x) \rvert < \infty$;
		\item \label{condlast} there exist $C \geq 1$ and $\gamma > 0$ such that for all $x \in (0,r_1)$,
		\begin{equation*}
			C^{-1} \leq \frac{f''(x)}{x^{\gamma-1}} \leq C.
		\end{equation*}
	\end{enumerate}
	
	\begin{remark}
		The implicit assumption in Condition \ref{cond3} that $\lvert f'_+(r_1)\rvert>1$ is used only to ensure the required expansion and distortion properties of the first return map to $[r_1,1]$. Thus the results below remain valid if one allows $\lvert f'_+(r_1)\rvert=1$, provided that the induced first return map is uniformly expanding and satisfies the bounded distortion estimates used in the inducing and transfer arguments below.
	\end{remark}
	
	\begin{remark}[Endpoint conventions]\label{rem:endpointrelations}
		We use the interval representation of Manneville-Pomeau maps. The particular values assigned at branch endpoints $f(0),f(r_1),f(1)$ are arbitrary conventions. Changing them affects only endpoint and preimage orbits, a countable set treated separately below, and does not change our strong-winning conclusion.
	\end{remark}
	
	The standard example from this class is the Manneville--Pomeau map
	\begin{equation*}
		f(x) = \begin{cases}
			x+x^{1+\gamma} & \text{if } \, 0 \leq x < r_1 \\
			x+x^{1+\gamma} - 1 & \text{if } \, r_1 \leq x \leq 1,
		\end{cases}
	\end{equation*}
	where $r_1$ is the unique solution of $x+x^{1+\gamma} = 1$.

	\subsection{First return to $[r_1,1]$ and the induced map}
	
	For any $x \in [0,1]$ define the first return time
	\begin{equation*}
		\tau (x) := \min \{ n \geq 1 \colon f^n (x) \in [ r_1,1 ] \}.
	\end{equation*}
	The induced first return map $F \colon [r_1,1] \to [r_1,1]$ is then given by
	\begin{equation*}
		F(x) := f^{\tau (x)} (x).
	\end{equation*}
	Condition \ref{cond3} above implies that $F$ is uniformly expanding. Define
	\begin{equation*}
		\lambda \coloneqq \inf \lvert F' \rvert > 1,
	\end{equation*}
	using one-sided derivatives as needed.

	\subsection{Schmidt's game and the strong variation}\label{sec:schmidt}
	
	We use the strong variation of Schmidt's game introduced by McMullen in \cite{McMullen2010}, specialized to our setting. Let $I \subset \mathbb{R}$ be a compact interval and fix parameters $\alpha,\beta \in (0,1)$ and a set $S \subset I$. Bob chooses any compact interval $B_1 \subset I$. Thereafter the players alternately choose nested compact intervals $B_1 \supset A_1 \supset B_2 \supset A_2 \supset \dots$ satisfying $\lvert A_k \rvert \geq \alpha \lvert B_k \rvert$ and $\lvert B_{k+1} \rvert \geq \beta \lvert A_k \rvert$ for $k \in \mathbb{N}$. Alice wins if
	\begin{equation*}
		\bigcap_{k=1}^{\infty} B_k \cap S\neq\emptyset.
	\end{equation*}
	
	If Alice is able to win for a fixed $\alpha$ and every $\beta \in (0,1)$, we say that $S$ is $\alpha$-strong winning. If this holds for some $\alpha>0$, we say that $S$ is strong winning. Strong-winning sets are dense, uncountable, have full Hausdorff dimension, and countable intersections of $\alpha$-strong winning sets are again $\alpha$-strong winning. Cocountable subsets of $\alpha$-strong winning sets are $\alpha$-strong winning, and the strong winning property is stable under quasisymmetric homeomorphisms. See \cite{Schmidt1966,McMullen2010,DuvallMP} and the references therein for proofs and background.

	\subsection{A finite-branch result}
	
	We will use the following finite-branch theorem of Hu--Li--Yu.
	
	\begin{thm}[Hu--Li--Yu \cite{HuLiYu2025}]\label{thm:HLY}
		Let $T\colon X\to X$ be a piecewise $C^{1+\delta}$ expanding interval map with finitely many branches on a compact interval $X$. Then, for every $\xi\in X$, the exceptional set \begin{equation*}
			BAD_T(\xi) = \{x\in X:\inf_{n\geq0}|T^n x-\xi|>0\}
		\end{equation*}
		is $1/2$-strong winning.
	\end{thm}
	
	Theorem \ref{thm:HLY} will be applied to the finite-branch approximating maps $F_K$ constructed in \S\ref{sec:transferfromF}. Since the induced map $F$ has infinite branches, the theorem cannot be applied directly to $F$.

	\section{Exceptional sets of the induced map}\label{sec:transferfromF}
	
	We now prove that the induced first-return map $F\colon [r_1,1]\to[r_1,1]$ has strong-winning exceptional sets with uniform winning parameter independent of the target. Our approach is to approximate $F$ by finite-branch expanding maps and then use Theorem \ref{thm:HLY}.
	
	Recall that $r_1$ is the branch boundary point of the map $f$. Define $r_0\coloneqq1$, and for $i\ge2$ define
	\begin{equation*}
		r_i\coloneqq f^{-1}\vert_{[0,r_1]}(r_{i-1}).
	\end{equation*}
	Number the branches $J_i^F$ of the induced map $F$ from right to left and write
	\begin{equation*}
		J_i^F=[p_i,p_{i-1}],
	\end{equation*}
	so that $p_0=1$ and $p_i\searrow r_1$.
	
	For each $K\ge1$, define $F_K\colon [r_1,1]\to [r_1,1]$ by
	\begin{equation*}
		F_K(x)=
		\begin{cases}
			F(x), & x>p_K,\\[3pt]
			r_1+\dfrac{1-r_1}{p_K-r_1}(x-r_1), & x\le p_K.
		\end{cases}
	\end{equation*}
	Thus $F_K$ agrees with $F$ on the first $K$ branches and replaces the infinite-branch tail $[r_1,p_K]$ with a single affine branch.
	
	Since $F$ is uniformly expanding and the slope of $F_K$ on branch $K+1$ exceeds 1, $F_K$ is uniformly expanding. And $F_K$ is piecewise $C^2$ since $f$ and hence $F$ are. Therefore $F_K$ satisfies the requirements of Theorem \ref{thm:HLY}, proving the following result.
	
	\begin{lem}\label{lem:FKHLY}
		For every integer $K\geq 1$ and $q\in[r_1,1]$, the exceptional set $\mathcal E_{F_K}(q)$ is $1/2$-strong winning.
	\end{lem}
	
	Now we show that winning strategies with the finite-branch maps $F_K$ transfer cleanly to the infinite-branch $F$.	
	
	\begin{thm}\label{thm:Fwinning}
		There exists a common winning parameter $\alpha_F\in\bigl(0,\frac12\bigr)$ such that for every $q\in[r_1,1]$, the exceptional set $\mathcal E_F(q)$ is $\alpha_F$-strong winning.
	\end{thm}
	
	\begin{proof}
		Theorem 2.1 of \cite{DuvallMP} states that $\mathcal E_F(r_1)$ is $\alpha_1$-strong winning for some $\alpha_1\in \bigl(0,\frac12\bigr)$. Now fix $q\in(r_1,1]$.
		
		By Lemma \ref{lem:FKHLY}, for every $K\ge1$, the set $\mathcal E_{F_K}(q)$ is $1/2$-strong winning and hence $\alpha_1$-strong winning. Therefore
		\begin{equation*}
			S\coloneqq \mathcal E_F(r_1)\cap\bigcap_{K=1}^\infty \mathcal E_{F_K}(q)
		\end{equation*}
		is $\alpha_1$-strong winning.
		
		We claim that $S\subset \mathcal E_F(q)$. Let $x\in S$. Since $x\in\mathcal E_F(r_1)$, the $F$-orbit of $x$ misses a relative neighborhood of $r_1$ in $[r_1,1]$. Hence there exists $N\ge1$ such that every point in the $F$-orbit of $x$ lies in the union of the first $N$ branches of $F$. On this orbit, $F_N$ and $F$ agree. Since $x\in\mathcal E_{F_N}(q)$, the $F_N$-orbit of $x$ misses a relative neighborhood of $q$. The same is therefore true for the $F$-orbit of $x$, so $x\in\mathcal E_F(q)$.
		
		Since $S$ is $\alpha_1$-strong winning, so too is its superset $\mathcal E_F(q)$.
	\end{proof}
	
	\section{Exceptional sets of the original map}
	
	We now transfer the strong-winning result for the induced map \(F\) to the original map \(f\). We will need the following distortion estimate, which is Lemma 5 of \S6.2 in \cite{Young1999}.
	
	\begin{prop}[A distortion estimate for $f^n$] \label{prop:youngdist}
		There exists a constant $C_1 > 1$ such that for all integers $0 \leq m \leq n$, and for all points $x,y \in (r_{n+1}, r_n)$,
		\begin{equation*}
			\bigg\lvert \log \frac{( f^m )'x}{( f^m )'y} \bigg\rvert \leq \frac{C_1}{r_{n-m} - r_{n-m+1}} \big\lvert f^m x - f^m y \big\rvert.
		\end{equation*}
	\end{prop}
	
	\begin{prop}[Game transfer from $F$ to $f$]\label{prop:gametransferFtof}
		With $C_1>0$ as in Proposition \ref{prop:youngdist}, for every $\alpha \in \bigl( 0, \frac12 \bigr]$:
		\begin{enumerate}[label=(\roman*)]
			\item\label{claim1} If $\mathcal{E}_F(r_1)$ is $\alpha$-strong winning, then $\mathcal E_f(0)$ is $e^{-C_1} \alpha$-strong winning;			
			\item\label{claim2} If $p \notin \bigcup_{n=0}^\infty f^{-n}(0)$ and $\mathcal{E}_F\left( f^{\tau(p)}p \right)$ is $\alpha$-strong winning, then $\mathcal E_f(p)$ is $e^{-C_1} \alpha$-strong winning.
		\end{enumerate}
		
	\end{prop}
	
	\begin{proof}
		Statement \ref{claim1} is Theorem 2.1 of \cite{DuvallMP}. For Statement \ref{claim2}, fix $p \notin \bigcup_{n=0}^\infty f^{-n}(0)$ and suppose that $\mathcal{E}_F\left( f^{\tau(p)}p \right)$ is $\alpha_F$-strong winning ($0<\alpha_F\leq\frac12$). Define
		\begin{equation*}
			\alpha_f \coloneqq e^{-C_1} \alpha_F \in \bigl(0,\tfrac12\bigr)
		\end{equation*}
		and let $\beta_f\in(0,1)$ be arbitrary. Let
		\begin{equation*}
			\beta_F\coloneqq e^{-C_1}\beta_f \in (0,1).
		\end{equation*}
		
		We play two strong Schmidt games. Alice and Bob play the primary $(\alpha_f,\beta_f)$-strong game on $[0,1]$ with target set $\mathcal{E}_f(p)$, while Alicia and Bobby play an auxiliary $(\alpha_F,\beta_F)$-strong game on $[r_1,1]$ with target set $\mathcal{E}_F \big( f^{\tau(p)}p\big)$.
		
		Bob begins the main game by choosing a closed interval $B_1\subset[0,1]$. Alice chooses any permissible $A_1 \subset B_1$ not containing 0. Bob then chooses $B_2\subset A_1$. Alice plays arbitrarily until Bob chooses an interval contained in the interior $(r_{N+1},r_N)$ of a return branch. This occurs after finitely many turns: once $0$ has been excluded, only finitely many return branches meet Bob's current interval, and Alice can force the lengths of Bob's intervals to decrease to zero while avoiding the finitely many branch endpoints since $\alpha_f < \frac12$. After relabeling we may assume without loss of generality that
		\begin{equation*}
			B_1\subset (r_{N+1},r_N)
		\end{equation*}
		for some integer $N \geq 0$. On this branch, $f^N$ maps $(r_{N+1},r_N)$ diffeomorphically onto $(r_1,1)$ (where $F^0$ is understood to mean the identity map).
		
		The auxiliary game begins when Bobby chooses the compact interval
		\begin{equation*}
			B_1'\coloneqq f^N(B_1)\subset (r_1,1).
		\end{equation*}
		Alicia then chooses $A_1'\subset B_1'$ according to a winning strategy. Alice now chooses the compact interval
		\begin{equation*}
			A_1\coloneqq (f^N)^{-1}(A_1')\cap (r_{N+1},r_N) \subset B_1.
		\end{equation*}
		This is a legal move because, using the Mean Value Theorem and Proposition \ref{prop:youngdist},
		\begin{equation}\label{eq:A1legal}
			\frac{\lvert A_1\rvert}{\lvert B_1\rvert}
			\geq
			e^{-C_1}\frac{\lvert A_1'\rvert}{\lvert B_1'\rvert}
			\geq
			e^{-C_1}\alpha_F
			=
			\alpha_f.
		\end{equation}
		Suppose that for some $k\geq1$ the players have chosen their intervals $\{ B_i,A_i,B_i',A_i' \}_{i=1}^k$ such that
		\begin{equation*}
			B_k'=f^N(B_k)
			\qquad\text{and}\qquad
			A_k=(f^N)^{-1}(A_k')\cap (r_{N+1},r_N),
		\end{equation*}
		and where Alicia has chosen her intervals $A_i'$ according to a winning strategy. Bob chooses $B_{k+1}\subset A_k$. Bobby chooses the compact interval
		\begin{equation*}
			B_{k+1}'\coloneqq f^N(B_{k+1})\subset f^N(A_k) \subset A_k';
		\end{equation*}
		as before this is a legal move because, again using the Mean Value Theorem and Proposition \ref{prop:youngdist},
		\begin{equation*}
			\frac{\lvert B_{k+1}'\rvert}{\lvert A_k'\rvert}
			\geq
			e^{-C_1}\frac{\lvert B_{k+1}\rvert}{\lvert A_k\rvert}
			\geq
			e^{-C_1}\beta_f
			=
			\beta_F.
		\end{equation*}
		Alicia chooses $A_{k+1}'\subset B_{k+1}'$ according to a winning strategy. Alice chooses
		\begin{equation*}
			A_{k+1}\coloneqq (f^N)^{-1}(A_{k+1}')\cap (r_{N+1},r_N);
		\end{equation*}
		the obvious modification to \eqref{eq:A1legal} proves that this is a legal move.
		
		Both games are complete, and by construction Alicia wins the auxiliary game. Let
		\begin{equation*}
			q\coloneqq f^{\tau(p)}(p).
		\end{equation*}
		Since Alicia wins, there exists
		\begin{equation*}
			\omega'\in \bigcap_{k=1}^{\infty}B_k'\cap \mathcal E_F(q).
		\end{equation*}
		Define
		\begin{equation*}
			\omega\coloneqq (f^N)^{-1}\vert_{[r_{N+1},r_N]} (\omega').
		\end{equation*}
		Because \(B_k'=f^N(B_k)\) for every \(k\), it follows that
		\begin{equation*}
			\omega\in \bigcap_{k=1}^{\infty}B_k.
		\end{equation*}
		It remains to show that \(\omega\in\mathcal E_f(p)\).
		
		Since \(\omega'\in \mathcal E_F(q)\), the \(F\)-orbit of \(\omega'\) avoids some relative open neighborhood \(V\) of \(q\) in \([r_1,1]\). We claim that the \(f\)-orbit of \(\omega\) avoids a relative neighborhood of \(p\).
		
		First note that the finite orbit segment
		\begin{equation*}
			p,\ f(p),\ldots, f^{\tau(p)-1}(p)
		\end{equation*}
		does not contain the branch point \(r_1\) since $p$ is not a preimage of 0. Hence there exists a relative open neighborhood \(U_0\) of \(p\) such that, for every \(0\leq i<\tau(p)\), the interval \(f^i(U_0)\) remains in the same monotonicity branch of \(f\) as \(f^i(p)\). Consequently \(f^{\tau(p)}\) is continuous on \(U_0\). Thus find a relative open neighborhood \(U\subset U_0\) of \(p\) such that
		\begin{equation*}
			f^{\tau(p)}(U)\subset V.
		\end{equation*}
		
		Suppose, toward a contradiction, that the \(f\)-orbit of \(\omega\) enters \(U\). Then there exists \(t\geq0\) such that
		\begin{equation*}
			f^t(\omega)\in U.
		\end{equation*}
		Therefore
		\begin{equation*}
			f^{t+\tau(p)}(\omega)\in V\subset [r_1,1].
		\end{equation*}
		Since \(\omega\in (r_{N+1},r_N)\), its first return to \([r_1,1]\) occurs at time \(N\), and
		\begin{equation*}
			f^N(\omega)=\omega'.
		\end{equation*}
		Because \(f^{t+\tau(p)}(\omega)\in [r_1,1]\), necessarily \(t+\tau(p)\geq N\). Moreover, after the first return at time \(N\), every later visit of the \(f\)-orbit of \(\omega\) to \([r_1,1]\) is obtained by iterating the first return map \(F\). Hence there exists \(j\geq0\) such that
		\begin{equation*}
			f^{t+\tau(p)}(\omega)=F^j(\omega').
		\end{equation*}
		This contradicts the choice of \(V\), since the \(F\)-orbit of \(\omega'\) avoids \(V\). Thus the \(f\)-orbit of \(\omega\) avoids \(U\), and so
		\begin{equation*}
			\omega\in \mathcal E_f(p).
		\end{equation*}
		
		This shows that
		\begin{equation*}
			\bigcap_{k=1}^{\infty}B_k\cap \mathcal E_f(p)\neq\emptyset.
		\end{equation*}
		and therefore Alice wins the main \((\alpha_f,\beta_f)\)-strong game. Since \(\beta_f\in(0,1)\) was arbitrary, \(\mathcal E_f(p)\) is \(e^{-C_1}\alpha_F\)-strong winning.
	\end{proof}
	
	Now we address the case when the point to be avoided is a preimage of 0.
	
	\begin{lem}[Preimages of the neutral fixed point]\label{lem:preimagesofzero}
		For every
		\begin{equation*}
			p\in\bigcup_{n=0}^{\infty}f^{-n}(0),
		\end{equation*}
		we have
		\begin{equation*}
			\mathcal E_f(0)\cap\mathcal E_f(1)\subset \mathcal E_f(p).
		\end{equation*}
	\end{lem}
	
	\begin{proof}
		Suppose $x\notin\mathcal E_f(p)$. Then there are integers $n_j\to\infty$ such that $f^{n_j}x\to p$. Choose $N\geq0$ with $f^N(p)=0$.
		
		The map $f^N$ is continuous on each component of the complement of the finite set
		\begin{equation*}
			\bigcup_{i=0}^{N-1} f^{-i}(r_1).
		\end{equation*}
		Since $f(0)=0$, $f(r_1)=0$, the left branch tends to $1$ at $r_1$, and $f(1)=1$, every one-sided limit of $f^N(y)$ as $y\to p$ is either $0$ or $1$. Passing to a subsequence of $f^{n_j}x$ lying on one side of each relevant discontinuity of $f^N$, we obtain either
		\begin{equation*}
			f^{n_j+N}x\to0
		\end{equation*}
		or
		\begin{equation*}
			f^{n_j+N}x\to1.
		\end{equation*}
		Thus $x\notin\mathcal E_f(0)$ or $x\notin\mathcal E_f(1)$.
	\end{proof}
	
	Our main theorem is now a simple consequence of Proposition \ref{prop:gametransferFtof} and Lemma \ref{lem:preimagesofzero}.
	
	\begin{proof}[Proof of Theorem \ref{thm:maintheorem}]
		The exceptional set $\mathcal E_f(0)$ is $\alpha_0$-strong winning for some $\alpha_0\in\bigl(0,\frac12\bigr)$ by Theorem 1.2 of \cite{DuvallMP}. For $p \in (0,1]$ not a preimage of 0, Theorem \ref{thm:Fwinning} and Proposition \ref{prop:gametransferFtof} imply that $\mathcal E_f(p)$ is $\alpha_1$-strong winning for some $\alpha_1\in\bigl(0,\frac12\bigr)$. Also, 1 is not a preimage of 0, so in particular $\mathcal E_f(1)$ is $\alpha_1$-strong winning. Finally, for $p\in\bigcup_{n=0}^{\infty}f^{-n}(0)$, Lemma \ref{lem:preimagesofzero} now implies that $\mathcal E_f(p)$ is $\min\{\alpha_0,\alpha_1\}$-strong winning. Thus Theorem \ref{thm:maintheorem} follows with
		\begin{equation*}
			\alpha\coloneqq \min\{\alpha_0,\alpha_1\}.\qedhere
		\end{equation*}
	\end{proof}

\end{document}